\documentclass[11pt]{amsart}

\usepackage{mathrsfs}
\usepackage[T1]{fontenc}
\usepackage{amsfonts,amssymb,amsthm, txfonts, pxfonts,amscd}
\usepackage[stretch=10]{microtype}
\usepackage{epsfig}
\usepackage{soul,color}
\usepackage{hyperref}
\def\struckint{\mathop{%
\def\mathpalette##1##2{\mathchoice{##1\displaystyle##2}%
 {##1\textstyle##2}{##1\scriptstyle##2}{##1\scriptscriptstyle##2}}%
\mathpalette
{\vbox\bgroup\baselineskip0pt\lineskiplimit-1000pt\lineskip-1000pt
\halign\bgroup\hfill$}
{##$\hfill\cr{\intop}\cr\diagup\cr\egroup\egroup}%
}\limits}

\def\inf{\mathop{\rm inf}\nolimits}
\def\sup{\mathop{\rm sup}\nolimits}

\def\stochk{\mathop{\mathcal{S}_k}}
\def\simplexk{\mathop{\Delta_k}\nolimits}

\def\part{\mathop{\mbox{Part}}\nolimits}

\def\ell{\mathop{[l]}\nolimits}
\def\equalinlaw{\mathop{=_{\mathcal{L}}}\nolimits}
\newcommand{\zz}[1]{\mathbb #1}

\newcommand{\xnorm}[1]{ \Vert #1 \Vert } 

\newtheorem{thm}{Theorem}[section]

\newtheorem{prop}[thm]{Proposition}
\newtheorem{cor}[thm]{Corollary}

\theoremstyle{definition}
\newtheorem{defn}[thm]{Definition}
\newtheorem{example}[thm]{Example}

\newtheorem{note}[thm]{Note}
\newtheorem{rmk}[thm]{Remark}

\title[Exchangeable Markov Processes]{Exchangeable
Markov Processes on $[k]^{\zz{N}}$ with Cadlag Sample Paths}
\author{Harry Crane}\address{Rutgers University\\ Department
of Statistics \\ 
110 Frelinghuysen Road\\
Piscataway, NJ 08854.}
\author{Steven P. Lalley}\address{University of Chicago\\
Department of Statistics\\
Eckhart Hall\\
5734 S.\ University Ave.\\
Chicago, IL 60637}

\date{\today}
\thanks{First author supported by NSF grant DMS-1308899 and NSA grant H98230-13-1-0299.  Second author supported by  NSF grant DMS-1106669}
\begin{document}
\maketitle
\begin{abstract}
Any exchangeable, time-homogeneous Markov processes on
$[k]^{\mathbb{N}}$ with cadlag sample paths projects to a Markov
process on the simplex whose sample paths are cadlag and of locally
bounded variation.  Furthermore, any such process has a de
Finetti-type description as a mixture of i.i.d.\ copies of
time-inhomogeneous Markov processes on $[k]$. In the Feller case,
these time-inhomogeneous Markov processes have a relatively simple
structure; however, in the non-Feller case a greater variety of
behaviors is possible since the transition law of the underlying
Markov process on $[k]^{\zz{N}}$ can depend in a non-trivial way on
the exchangeable $\sigma$-algebra of the process.
\end{abstract}

\section{Introduction}\label{section:introduction} 

We study exchangeable, time-homogeneous Markov processes with cadlag
sample paths on $[k]^{\zz{N}}$, the space of $k$-colorings of the
natural numbers $\zz{N}$.  See Section~\ref{ssec:preliminaries} for a
formal definition.  The state space is endowed with the
product-discrete topology, i.e.\ the topology of point-wise
convergence, under which it is compact and metrizable.  In previous
work, Crane \cite{Crane2012} characterized the class of exchangeable,
continuous-time Markov processes having the Feller property on
$[k]^{\zz{N}}$ with respect to the product-discrete topology; see
Section~\ref{section:feller} for a r\'esum\'e of Crane's results and a
discussion of their relation to the main results of this paper.
Exchangeable Markov processes with the Feller property are precisely
those exchangeable Markov processes that are \emph{consistent under
sub-sampling}, that is, for each $n\in \zz{N}$ the projection on the
first $n$ coordinates is itself Markov. Conversely, by Kolmogorov's extension
theorem  (\cite{Billingsley1995}, Chapter~7, Section~37) any consistent
family of exchangeable Markov processes on the finite state spaces
$[k]^{[n]}$ extends naturally to an exchangeable Markov process on
$[k]^{\zz{N}}$ with the Feller property. Another well-known family of exchangeable
partition-valued Feller processes is the class of exchangeable
fragmentation-coalescence (EFC) processes, which have been
characterized by Berestycki \cite{Berestycki2004}, cf.\ Bertoin
\cite{Bertoin2006} and Pitman \cite{Pitman2005}.  An EFC process can
be expressed as a superposition of independent, exchangeable, coalescent
and fragmentation processes.  The classical example of an
exchangeable partition-valued processes, the exchangeable coalescent
\cite{Kingman1982}, was originally introduced in connection to certain
population genetics models; processes on $[k]^{\zz{N}}$ also arise
somewhat naturally in genetics, as discussed in \cite{Crane2012}.

The objective of this paper is to characterize exchangeable Markov
processes on $[k]^{\zz{N}}$ that do not necessarily satisfy the Feller
property, but do have cadlag sample paths.  The characterization we give
is related to that in \cite{Crane2012}, but, because the processes are
not assumed to possess the Feller property, we cannot draw general
conclusions about the infinitesimal behavior of these processes.  In
particular, the infinitesimal generator of a non-Feller process need
not exist and, therefore, in general, no characterization in terms of
infinitesimal jump rates can be given. As we will show in
Section~\ref{ssec:mean-field} below, \emph{non-Feller} exchangeable
Markov processes on $[k]^{\zz{N}}$ arise naturally as limits of
various interesting \emph{mean-field} models, including the
\emph{mean-field stochastic Ising model} (cf., e.g., \cite{peres}) and
the \emph{Reed-Frost stochastic epidemic} (cf. e.g., \cite{daley-gani},
\cite{dolgoarshinnykh-lalley}). These examples show that non-Feller exchangeable
Markov processes are by no means pathological; in many circumstances
the non-Feller case is the one of greatest {practical} interest.

By viewing elements of $\mathbb{N}$ as distinct particles, one can the
regard processes on $[k]^{\mathbb{N}}$ as interacting particle systems
in which each particle is in one of $k\geq1$ internal states at every
time.  The assumption that such a process has cadlag paths is
equivalent to the assumption that, with probability one, each particle
spends a strictly positive amount of time in each internal state it
visits. When the Feller property fails, the evolution of any subset of
particles can depend on the configuration of the entire collection
$\mathbb{N}$, but, for exchangeable processes, only via the
exchangeable $\sigma$-algebra.

The paper is organized as follows.  In Section
\ref{section:preliminaries}, we formally define exchangeable Markov
processes, present some illustrative examples, and state the
main results of the paper. In Section \ref{section:discrete-time}, we
discuss properties of the joint behavior of processes observed at an
at most countable collection of times, which includes discrete-time
Markov chains. In Sections
\ref{section:color-swatch-process}-\ref{section:construction}, we
discuss various properties of these processes in continuous time, and
give the proofs of the main Theorems \ref{thm:main},
\ref{thm:semigroup}, and \ref{thm:k2}. Finally, in
Section~\ref{section:feller}, we indicate briefly how the
representation of Theorem~\ref{thm:semigroup} specializes when the
underlying Markov process on $[k]^{\zz{N}}$ is Feller.

\noindent 
\textbf{Notation.} Elements of the state space  $[k]^{\zz{N}}$, $[k]:=\{1,\ldots,k\}$,
 will be denoted by $x= (x^{n})_{n\in \zz{N}}$, with \emph{superscripts}
indicating coordinates. Paths in $[k]^{\zz{N}}$ will be denoted by
$(x_{t})_{t\in J}$, where $J$ is a subset of $[0,\infty)$; in
particular, \emph{subscripts} will be used to denote \emph{time}. When
convenient, we use the abbreviation $x_{J}$ for the path
$(x_{t})_{t\in J}$, or similarly $x^{i}_{J}$ for the $i$th coordinate
path $(x^i_t)_{t\in J}$. A path $(x_{t})_{t\in J}$ in
$[k]^{\zz{N}}$ is said to be \emph{cadlag} if each coordinate path
$(x^{i}_{t})_{t\in J}$ is piecewise constant and
right-continuous.

\section{Main Results}\label{section:preliminaries} 

\subsection{Exchangeability}\label{ssec:preliminaries} Let $S$ be a
finite or countable set of \emph{sites}. A \emph{permutation}
of $S$ is a bijection $\sigma :S: \rightarrow S$ that fixes all but
finitely many elements of $S$. Any permutation of $S$ induces a
permutation of $[k]^{S}$ by relabeling, that is, for any element $x=
(x^{i})_{i\in S}$ of $[k]^{S}$ the image $x^{\sigma}=\sigma x$ is defined by
\[
	(\sigma x)^{i}=x^{\sigma (i)}.
\]
We extend this notation to measurable subsets $A\subset [k]^{S}$
by writing $A^{\sigma}=(\sigma A):=\{\sigma x \,:\, x\in A\}$. A random variable
$X$ taking values in the space $[k]^{S}$ is said to be
\emph{exchangeable} (or to have an \emph{exchangeable distribution})
if, for every permutation $\sigma$, the distribution of $\sigma X$
coincides with that of $X$.

 \begin{defn}\label{definition:exchangeableMP} A Markov process
$(X_{t})_{t\geq0}$ in $[k]^{S}$ (either in
discrete or continuous time) is said to be \emph{exchangeable} if
\begin{enumerate}
\item [(a)] its initial state $X_{0}$ is  exchangeable, and
\item [(b)] its transition kernel $p_{t} (x,dy)$ is invariant under
permutations $\sigma$ of $S$; that is, for every measurable subset
$A\subseteq[k]^{S}$,
\begin{equation}\label{eq:exchangeableKernel}
	p_{t} (\sigma x, \sigma A)=p_{t}(x,A)\quad\mbox{for all }t\geq0.
\end{equation}
\end{enumerate}
\end{defn}

 Recall that a transition kernel $p_{t} (x,dy)$ has the
\emph{Feller property} if, for every bounded, continuous, real-valued
function $f$ on the state space, the function
\[
	T_{t}f (x):=\int f (y)\,p_{t} (x,dy)
\]
is jointly continuous in $(t,x)$.   Our main results
provide a description of Markov processes with cadlag sample paths
whose transition kernels are exchangeable but may fail to satisfy the
Feller property; however, our description will only apply to those
Markov processes $(X_{t})_{t\geq 0}$ whose initial states $X_{0}$ are
distributed according to an exchangeable probability measure on
$[k]^{\zz{N}}$.

Except in Section \ref{ssec:mean-field}, we will be concerned
exclusively with the case $S=\zz{N}$. In this case, for any finite
subset or interval $A\subseteq [0,\infty )$ we denote by
$\mathcal{E}_{A}$ the exchangeable $\sigma$-algebra for the sequence
$X_A:=(X_A^1,X^{2}_{A},\dotsc)$, and we abbreviate
$\mathcal{E}=\mathcal{E}_{[0,\infty )}$.  In particular,
$\mathcal{E}_A$ consists of all events $E$ that are both measurable
with respect to $X_A$ and satisfy $E=E^{\sigma}$, for all permutations
$\sigma:\mathbb{N}\rightarrow\mathbb{N}$ fixing all but finitely many
$n\in\mathbb{N}$.

\subsection{Example: Mean-Field Models}\label{ssec:mean-field}

If $\{X_{t} \}_{t\geq 0}$ is an exchangeable Markov process on a
finite configuration space $[k]^{S}$ then it projects naturally to a
Markov process on the $k$-simplex $\simplexk$, via the mapping
\begin{equation}\label{eq:finite-proj}
	Y^{i}_{t} =\frac{1}{|S|} \sum_{s\in S} \mathbf{1}\{X^{s}_{t}=i \}.
\end{equation}
Moreover, the transition rates for the Markov process $\{X_{t}
\}_{t\geq 0}$ depend only on the current state of the projection
$Y_{t}$.

\begin{example}\label{example:ising}
A \emph{mean-field stochastic Ising model} (also known as
\emph{Glauber dynamics}) is an exchangeable Markov
process on the finite state space $\{-1,+1 \}^{[n]}$ whose
infinitesimal transition rates are determined by a \emph{Hamiltonian
function} $H:\{-1,+1 \}^{[n]} \rightarrow \zz{R}$ that is invariant
under permutations $\sigma :[n] \rightarrow [n]$, that is, $H\circ
\sigma =H$. Only one spin may flip at a time, and for any two
configurations $x,x'\in \{-1,+1 \}^{[n]}$ that differ in only one
coordinate the infinitesimal jump rate from $x$ to $x'$ is given by
\begin{equation}\label{eq:ising}
	P \{X_{t+dt}=x'\,|\, X_{t}=x\}=e^{-\beta H
	(x')} \,dt,
\end{equation}
where $\beta >0$ is the inverse temperature parameter.  Since $H$ is
assumed to be invariant under permutations, the Markov process with
transition rates \eqref{eq:ising} is exchangeable. In the case of
greatest physical interest, the Hamiltonian function $H$ is determined
by (i) interactions between the spins and an external field and (ii)
pair interactions between spins. In this case, the transition rates
depend only on the projections $Y_{t}$ to the simplex (that is, only
on the proportion of sites labeled $+1$). 

\end{example}

\begin{example}\label{example:reed-frost}
The continuous-time \emph{Reed-Frost epidemic model} is a simple
mean-field model of an epidemic in which members of a population $[n]$
can be either \emph{susceptible, infected,} or \emph{recovered}
(henceforth abbreviated $S,I,R$, respectively).  Each susceptible
individual becomes infected at rate proportional to the total number
currently infected; each infected individual recovers at a rate
$\varrho$ not depending on the conditions of the other individuals;
and recovered individuals are granted permanent immunity from further
infection, hence remain recovered forever and play no further role in
the propagation of the epidemic. Thus, for
certain parameters $\beta ,\varrho >0$, for each individual $i\in
[n]$,
\begin{gather}
\label{eq:reed-frost}	
	 P (X^{i}_{t+dt}=I\,|\, X^{i}_{t}=S \; \text{and}\;
	N^{I}_{t}=m) = \beta m \,dt;\\
\notag 	P (X^{i}_{t+dt}=R\,|\, X^{i}_{t}=I \;\text{and}\;
	N^{I}_{t}=m, \,N^{S}_{t}=p) =\varrho \,dt
\end{gather}
where
\[
	N^{j}_{t}= \sum_{i=1}^{n} \mathbf{1}\{X^{i}_{t}=j \} \quad \text{for}\;\; j=S,I,R.
\]
\end{example}

In both of these examples, the set $[n]$ of sites (or population members)
can be of arbitrary size, and in both cases the transition rates
extend to continuous functions on  the simplex $\simplexk$. It is
natural in circumstances such as these to inquire about large-$n$
limits. Clearly, such limits will exist only if the transition rates
are properly scaled in order that the flip rates at individual sites
remain bounded and nonzero as $n \rightarrow \infty$. Therefore, we
shall restrict attention to families of exchangeable Markov processes
on configuration spaces $[k]^{[n]}$, with $k$ fixed, whose transition rates are
obtained from a set of common (to all $n$) continuous functions
$f_{i,j}:\simplexk \rightarrow \zz{R}_{+}$ on the $k$-simplex.

\begin{defn}\label{definition:mean-field-model}
A \emph{mean-field model} with $k$ types is a family of exchangeable
Markov processes $X_{t}= (X^{i}_{t})_{i\in [n]}$ on the finite
configuration spaces $[k]^{[n]}$, one for each $n\in \zz{N}$, whose
transition rates are related in the following way: for each ordered
pair $(i,j)\in [k]$ there exists a continuous, nonnegative function
$f_{i,j}:\simplexk \rightarrow \zz{R}_{+}$ on the $k-$simplex $\simplexk$
such that for all $n\in \zz{N}$ and for each site $s\in [n]$,
\begin{equation}\label{eq:mean-field-model}
	P (X^{s}_{t+dt}=j\,|\, X_{t})= f_{i,j} (Y_{t}) \,dt \quad \text{on
	the event }\;\; X^{s}_{t}=i 
\end{equation}
where $Y_{t}$ is the projection of $X_{t}$ to $\simplexk$,  defined by
\eqref{eq:finite-proj}.
\end{defn}

\begin{prop}\label{proposition:mean-field-limits}
Consider a mean-field model with transition rates satisfying
\eqref{eq:mean-field-model}, where the functions $f_{i,j}$ are
nonnegative and Lipschitz continuous on $\simplexk$. Assume that the
initial states $X_{0}$ are exchangeable on $[k]^{[n]}$ and project to
points $Y_{0}\in \simplexk$ that converge as $n \rightarrow \infty$ to
some nonrandom point $y_{0}\in \simplexk $. Then, as $n \rightarrow
\infty$, the Markov processes $X_{t}$ on the finite configuration
spaces $[k]^{[n]}$ converge in distribution to an exchangeable Markov
process $\tilde{X}_{t}$ with cadlag sample paths on the infinite
configuration space $[k]^{\zz{N}}$.  The limit process is Feller if
and only if the rate functions $f_{i,j}$ are \emph{constant} on
$\simplexk$.
\end{prop}

The weak convergence of the Markov processes $X_{t}$ under these
hypotheses, at least in the case $k=2$, follows easily from classical
results of Kurtz \cite{kurtz} concerning the large-$n$ behavior of
density-dependent Markov processes. The general case $k\geq 2$ follows
by obvious and routine extensions of Kurtz's results.  Since the
important point is not the weak convergence but rather the fact that
the limit processes will, in general, be non-Feller, we shall only
sketch the proof in the special case $[k]=\{0,1 \}$.

\begin{proof}
[Proof of Proposition~\ref{proposition:mean-field-limits}]
Assume that $[k]=\{0,1 \}$. In this case, the simplex $\simplexk$ may
be identified with the unit interval $[0,1]$,  the rate functions 
$f_{0,1}$ and $f_{1,0}$ in
Definition~\ref{definition:mean-field-model} may be viewed as
functions on $[0,1]$ and 
\[
	Y_{t}=Y^{1}_{t}=\frac{1}{n}\sum_{s=1}^{n}\mathbf{1}\{X^{s}_{t}=1
	\}. 
\]
The Markov process $(Y_{t})_{t\in [0,1]}$ is an instance of what
Ethier and Kurtz (\cite{ethier-kurtz}, Chapter~11) call a
\emph{density-dependent} Markov process.  By Theorem~2.1 of
\cite{ethier-kurtz}, Chapter~11 (see also \cite{kurtz}), if the initial
states $Y_{0}$ converge to a point $y_0\in[0,1]$ as $n \rightarrow \infty$, 
then the processes $(Y_{t})_{t\geq 0}$ converge in
probability (on each finite time interval $[0,T]$) to a deterministic
trajectory $(y_{t})_{t\geq 0}$, where $(y_{t})_{t\geq 0}$ is the
unique solution to the differential equation
\begin{equation}\label{eq:ode}
	\frac{dy}{dt}=f_{0,1} (y_{t}) (1-y_{t})-f_{1,0} (y_{t})y_{t}
\end{equation}
with initial condition $y_{0}$. (Note: The hypothesis that the
functions $f_{i,j}$ are Lipschitz continuous is needed to guarantee
that solutions to the differential equation exist and are unique --
see condition (2.1) of \cite{ethier-kurtz}, Chapter~11.) Thus, for each
$T<\infty$ and $\varepsilon >0$,
\begin{equation}\label{eq:density-dependent-convergence}
	\lim_{n \rightarrow \infty }P\{\max_{0\leq t\leq
	T}\xnorm{Y_{t}-y_{t}}\geq \varepsilon\} =0.
\end{equation}

Now consider the Markov processes $(X_{t})_{t\geq 0}$ on the state
spaces $[k]^{[n]}$. These can be constructed using independent
auxiliary Poisson processes $N^{s}_{i,j} (t)$ with rates
$\lambda_{i,j}=\max_{y\in [0,1]}f_{i,j} (y)$, one for each $s\in
\zz{N}$, and i.i.d. sequences of Uniform-[0,1] random variables
$U^{s}_{i,j} (m)$, according to the following rule: at the $m$th
occurrence time $\tau^{s}_{i,j} (m) $ of $N^{s}_{i,j}$, if site $s$
is currently of color $i$ then it flips to color $j$ if and only if
\[
	 U^{s}_{i,j} (m)\leq \frac{f_{i,j} (Y_{\tau^{s}_{i,j}
	 (m)})}{\lambda_{i,j}}. 
\]
Since the rate functions $f_{i,j}$ are continuous, it follows from
\eqref{eq:density-dependent-convergence} that 
\[
	\lim_{n \rightarrow \infty} f_{i,j} (Y_{\tau^{s}_{i,j}
	 (m)})=f_{i,j} (y_{\tau^{s}_{i,j}
	 (m)}). 
\]
Consequently, for each $K<\infty$ and each $T<\infty$, 
the processes $(X_{t}^{s})_{t\leq T,s\in [K]}$
converge jointly in law to the process $(\tilde{X}^{s}_{t})_{t\leq
T,s\in [K]}$  whose coordinate processes $(\tilde{X}^{s}_{t})_{t\leq
T}$ are independent, time-\emph{inhomogeneous} Markov processes
governed by the following rule: at the $m$th
occurrence time $\tau^{s}_{i,j} (m) $ of $N^{s}_{i,j}$, if site $s$
is currently of color $i$ then it flips to color $j$ if and only if
\begin{equation}\label{eq:limit-rules}
	 U^{s}_{i,j} (m)\leq \frac{f_{i,j} (y_{\tau^{s}_{i,j}
	 (m)})}{\lambda_{i,j}}. 
\end{equation}
This proves the weak convergence assertion of the proposition.

It remains to show that the limit process $(\tilde{X}_{t})_{t\geq 0}$
has the required properties. It is obvious that the sample paths of
$(\tilde{X}_{t})_{t\geq 0}$ are cadlag, and since the coordinate
processes are conditionally i.i.d.\ given the initial value $y_0$, it
follows that $(\tilde{X}_{t})_{t\geq 0}$ is exchangeable. It remains to
show that $(\tilde{X}_{t})_{t\geq 0}$ is a Markov process.  
 First, observe that the differential equations
\eqref{eq:ode} and the SLLN imply that for each time $t\geq 0$, with
probability one, the limiting  fraction of coordinates of
$\tilde{X}_{t}$ that have color $1$ exists and equals $y_{t}$, i.e.,
\[
	\lim_{n \rightarrow
	\infty}n^{-1}\sum_{s=1}^{n}\mathbf{1}\{\tilde{X}^{s}_{t}=1 \} =y_{t}.
\]
Since the process $(\tilde{X}_{t})_{t\geq 0}$ is built using
independent Poisson processes and Uniform-[0,1] random variables, and
since \eqref{eq:ode} defines a flow on $[0,1]$ (recall that the
functions $f_{i,j}$ are assumed to be Lipschitz, so solutions of
\eqref{eq:ode} are unique), it follows that $(\tilde{X}_{t})_{t\geq
0}$ is a Markov process on $[k]^{\zz{N}}$ with time-homogeneous
transition kernel. Finally, the Feller property can only hold if the
functions $f_{i,j}$ are constant, because otherwise the transition
probabilities will depend in a non-trivial way on the exchangeable
$\sigma$-algebra. 
\end{proof}                           

\begin{rmk}\label{remark:mean-field-model}
The structure of the limit process $(\tilde{X}_{t})_{t\geq 0}$ in
Proposition~\ref{proposition:mean-field-limits}  provides an
interesting example of a phenomenon that holds generally for
exchangeable Markov processes on $[k]^{\zz{N}}$, as we will show in
Theorem~\ref{thm:semigroup}: although the Markov process is
time-homogeneous, the coordinate processes $\tilde{X}^{s}_{t}$ are,
\emph{conditional on the exchangeable $\sigma$-algebra},
time-\emph{inhomogeneous} Markov processes.
\end{rmk}

\subsection{Failure of the Feller Property}\label{ssec:failure}
For exchangeable Markov processes without the Feller
property, the transition mechanism may depend in a nontrivial way on
its exchangeable $\sigma$-field, and so there is a greater variety of
possible behaviors than for Feller processes. Moreover,
in the absence of the Feller property, there need not exist an infinitesimal
generator (at least in the usual sense, as a densely defined linear
operator on a space of bounded, continuous functions), and so
infinitesimal jump rates need not exist. The following example
illustrates the difficulty.

\begin{example}\label{example:singularJumpRate}
Let\footnote{When $k=2$ we write $[k]=\{0,1 \}$ rather than
$[k]=\{1,2 \}$.} $[k]=\{0,1 \}$, and let $f:[0,1]\rightarrow [0,1]$ be
any increasing homeomorphism of the unit interval. Let
$U_{1},U_{2},\dotsc$ be independent, identically distributed (i.i.d.)
random variables with the Uniform distribution on $[0,1]$. For $i\in
\zz{N}$ and $0\leq t\leq 1$,  define
\begin{align*}
	X^{i}_{t}&=1 \quad\text{if} \;\; U_{i}\leq f (t),\\
	&=0 \quad\text{if} \;\; U_{i}> f (t).
\end{align*}
The process $(X_{t})_{0\leq t\leq1}$ is Markov and has cadlag sample
paths. Furthermore, $(X_t)_{t\geq0}$ is exchangeable. However, if the
homeomorphism $f$ has the property that its derivative is $0$ almost
everywhere (as would be the case if $f$ were the cumulative
distribution function of a singular probability measure on $[0,1]$
with dense support) then the jumps of $X_{t}$ would occur only at
times in a set of Lebesgue measure $0$.
\end{example}

Although the behavior of the above Markov process is, in certain
respects, pathological, the example nevertheless suggests what can
happen in general.  By the Glivenko-Cantelli Theorem, for each $t\geq
0$ the limiting frequency of 1s in the configuration $X_{t}$ exists
with probability one.  Consequently, for each $t$ there is a well-defined
projection $Y_{t}=\pi (X_{t})$ of the Markov process $X_{t}$ to the
$2$-simplex $[0,1]$. In Example~\ref{example:singularJumpRate}, the
projection $(Y_{t})_{t\geq0}$ is a purely deterministic function of $t$ that
follows an increasing, continuous, but not everywhere differentiable
trajectory from $0$ to $1$. Our first theorem shows that an
analogous projection exists in general, and furthermore that it exists
\emph{simultaneously for all $t$}, except possibly on a set of
probability zero.

\begin{thm}\label{thm:main}
Let $(X_{t})_{t\geq0}$ be an exchangeable
Markov process on $[k]^{\zz{N}}$ with cadlag sample paths. Then, with
probability one, for every $t\geq 0$ and every $j\in [k]$, the limiting
frequency 
\begin{equation}\label{eq:limFreqs}
	Y_t^j:=\lim_{n \rightarrow
	\infty}n^{-1}\sum_{m=1}^{n}\mathbf{1}\{X^{m}_{t} =j\}
\end{equation}
exists, and the vector $Y_{t}= (Y^{j}_{t})_{j\in [k]}$ is an element
of the $k$-simplex $\simplexk$ of probability distributions on
$[k]$. The process $(Y_{t})_{t\geq 0}$ is Markov and has cadlag sample
paths. Furthermore, the sample paths of $(Y_{t})_{t\geq 0}$ are of
locally bounded variation with respect to the $L^{1}$ metric on the
simplex $\simplexk$. 
\end{thm}

Except for the last -- and most interesting -- assertion, that the
sample paths of $Y_{t}$ are of locally bounded variation,
Theorem~\ref{thm:main} will be proved in
Sections~\ref{section:discrete-time}-\ref{section:color-swatch-process}:
see Corollary~\ref{corollary:cadlagProjections}. That the sample paths
are of locally bounded variation will be proved in
Section~\ref{sec:bv}. The converse assertion -- that every Markov
process on the simplex $\simplexk$ whose paths are cadlag and of
locally bounded variation is the projection of an exchangeable Markov
process on $[k]^{\zz{N}}$ -- is the content of Theorem~\ref{thm:k2}
below. 

\begin{example}\label{example:nonFellerProj}
This example shows that, in general, the projection $(Y_{t})_{t\geq0}$
need not be Feller. Let $[k]=\{0,1 \}$; thus, the simplex $\simplexk$
can be identified with the unit interval $[0,1]$. The transition
mechanism of the process $(X_{t})_{t\geq 0}$ will involve a sequence
$\tau_{1},\tau_{2},\dotsc$ of i.i.d.\ exponential random variables
with mean $1$, independent of the initial state, but will also
involve the projection $Y_{0}$ of the initial state. In order that
this be well-defined for all elements $x\in [k]^{\zz{N}}$, we extend
the definition of the projection as follows:
\begin{equation}\label{eq:extended-proj}
	Y_{0}:=\liminf_{n \rightarrow
	\infty}n^{-1}\sum_{m=1}^{n}\mathbf{1} \{X^{m}_{0}=1 \}.
\end{equation}
Given the initial state $X_{0}$ and the auxiliary random variables
$\tau_{j}$, define
\begin{align*}
	X^{i}_{t}&= \mathbf{1}_{[\frac{1}{2} ,1]} (Y_{0}) &\text{if} \;\; t&\geq \tau_{i},\\
	&=X^{i}_{0}  &\text{if} \;\; t&< \tau_{i}.
\end{align*}
If the initial state $X_{0}$ is exchangeable, then clearly the process
$(X_{t})_{t\geq 0}$ is Markov and exchangeable. Its projection $Y_{t}$
to the simplex $[0,1]$ follows a deterministic trajectory
$t\mapsto \gamma (Y_{0}, t)$, and, for each $Y_{0}\in (0,1)$, this trajectory
is continuous. However, the trajectories do not fit together to make a
(jointly) continuous flow: in particular, for every $t>0$ the function
$Y_{0}\mapsto \gamma (Y_{0}, t)$ is discontinuous at $Y_{0}=1/2$, and
therefore, the projection $Y_{t}$ is not Feller.
\end{example}

\begin{example}\label{example:nonUniqueness}
This example shows that in general the law of the projection
$(Y_{t})_{t\geq 0}$ does not uniquely determine the law of the
covering process $(X_{t})_{t\geq 0}$. The process $Y_{t}$ will be a
two-state chain that jumps back and forth between $p=1/3$ and $q=2/3$
at the occurrence times $T_{1}<T_{2}<\dotsb$ of a rate-1 Poisson
process.  We exhibit two exchangeable Markov processes on $\{0,1
\}^{\zz{N}}$, denoted $X_{t}$ and $Z_{t}$, both of which project to
$Y_{t}$. The first process $X_{t}$ evolves as follows: for each
$n=1,2,\dotsc $, (a) if $Y_{T_{n}-}=2/3$ then all of the sites $i$
for which $X^{i}_{T_{n}-}=1$ simultaneously toss fair coins to determine
their colors at time $T_{n}$, and (b) if $Y_{T_{n}-}=1/3$ then all of
the sites $i$ for which $X^{i}_{T_{n}-}=0$ simultaneously toss fair coins
to determine their colors at time $T_{n}$. The second process $Z_{t}$
evolves even more simply: at each time $T_{n}$, \emph{all} sites,
regardless of their current colors, toss $(1-Y_{T_{n}})$-coins to
determine their colors at time $T_{n}$.
\end{example}

\begin{example}\label{example:nonUniqueness-Feller}
Even if the Markov process $(X_t)_{t\geq0}$ is Feller, its law need not
be determined by its projection $(Y_t)_{t\geq0}$.
Take $[k]=\{0,1 \}$ and consider the degenerate process
$Y_t=(p,1-p)$ for all $t\geq0$, where $p<1/2$.  Then $(X_t)_{t\geq0}$ can
evolve either by
\begin{itemize}
	\item[(a)] remaining constant for all $t\geq0$; or \item[(b)]
at the occurrence times $T_1<T_2<\cdots$ of a rate-1 Poisson process, all of
the sites $i$ for which $X_{T_n-}^i=0$ flip   to $X_{T_n}^i=1$, while
each site $i$ with $X_{T_n-}^i=1$ flips to $X_{T_n}^i=0$ with
probability $p/ (1-p)$.
\end{itemize}
Observe that in both of these evolutions the projection
$(Y_{t})_{t\geq 0}$ into $\simplexk$ is degenerate at $(p,1-p)$.
\end{example}

The preceding  examples indicate that additional information is required to
\emph{uniquely} determine the law of the covering Markov process.
Because the paths of $Y_{t}$ are not necessarily differentiable, this
information cannot be encapsulated by infinitesimal rates. Thus, our
specification will involve an associated time-inhomogeneous Markov
semigroup on $[k]$. 

\begin{defn}\label{defn:cocycle}
Given a Markov process $\{Y_{t} \}_{t\geq 0}$ valued in the simplex
$\simplexk$, a \emph{compatible random semigroup} is a two-parameter family
$(Q_{s,t})_{0\leq s\leq t}$ of random $k\times k$ stochastic matrices that
satisfy the \emph{cocycle relation}
\begin{equation}\label{eq:cocycle}
	Q_{r,t}=Q_{r,s}Q_{s,t} \quad \text{for all} \;\; r\leq s\leq
	t\quad\text{a.s.},
\end{equation}
and the \emph{compatibility relations}
\begin{equation}\label{eq:q-y}
	Y_{s}Q_{s,t}=Y_{t}\quad\text{for all }s\leq t\quad\mbox{a.s.}
\end{equation}
\end{defn}

\begin{thm}\label{thm:semigroup}
Let $(X_{t})_{t\geq0}$ be an exchangeable
Markov process on $[k]^{\zz{N}}$ with cadlag sample paths, and let
$\{Y_{t} \}_{t\geq 0}$ be its projection to the simplex $\simplexk$. 
Then there exists a compatible random semigroup $(Q_{s,t})_{0\leq s\leq t}$
such that (i) for each pair $s\leq t$ the random  matrix
$Q_{s,t}$ is measurable with respect to the exchangeable 
$\sigma$-algebra $\mathcal{E}_{\{s,t\}}$; and (ii) conditional on the
exchangeable $\sigma$-algebra $\mathcal{E}_{[0,\infty )}$, the
coordinate processes $X^{i}_{t}$ are independent realizations of a
\emph{time-inhomogeneous} Markov chain on $[k]$ with transition probability
matrices $Q_{s,t}$, that is,
\begin{equation}\label{eq:IH-MC}
	P (X^{i}_{t}=b \,|\, X^{i}_{s}=a; \, \mathcal{E}_{[0,\infty
	)})=Q_{s,t} (a,b). 
\end{equation}
\end{thm}

Theorem \ref{thm:semigroup} follows from Corollary~\ref{corollary:stochasticMatrixProj}
in Section~\ref{section:discrete-time}. 
We emphasize that although the Markov process $(X_{t})_{t\geq
0}$ is time-homogeneous, its conditional distribution given the
exchangeable $\sigma$-algebra will in general be that of a
\emph{time-inhomogeneous} Markov process. 
Theorem~\ref{thm:semigroup} shows that the law of an
exchangeable Markov process on $[k]^{\zz{N}}$ is uniquely determined
by the law of the two-parameter matrix-valued process
$(Q_{s,t})_{s,t\geq 0}$.  Example~\ref{example:nonUniqueness} shows
that, in general, the law of $(Q_{s,t})_{s,t\geq 0}$ is not uniquely
determined by that of the projection $(Y_{t})_{t\geq 0}$. However, if
$(Y_{t})_{t\geq 0}$ is Markov and has cadlag sample paths of bounded
variation then there is at least one compatible semigroup.

\begin{thm}\label{thm:k2}
For every Markov process $(Y_{t})_{t\geq 0}$ on the k-simplex
$\simplexk$ whose sample paths are cadlag and of locally bounded
variation, there exists an exchangeable Markov process $(X_{t})_{t\geq
0}$ on $[k]^{\zz{N}}$ whose projection \eqref{eq:limFreqs} into
$\simplexk$ has the same distribution as $(Y_{t})_{t\geq 0}$.
\end{thm}

Theorem \ref{thm:k2} will be proved in
Section~\ref{section:construction} by constructing a compatible
random semigroup that moves the minimum possible amount of mass
subject to the compatibility condition \eqref{eq:q-y}.

\noindent 

\section{Discrete-Time Exchangeable Markov
Chains}\label{section:discrete-time}

\bigskip \emph{For the remainder of the paper, we assume that
$(X_{t})_{t\in T}$, where either $T=[0,\infty )$ or $T=\zz{Z}_{+}$, is
an exchangeable Markov process on $[k]^{\zz{N}}$ with cadlag sample
paths.}

\begin{prop}\label{proposition:exchangeableKernel}
For every finite sequence of
times $0=t_{0}<t_{1}<\dotsb <t_{n}<\infty$, the vector-valued sequence 
\[
\begin{pmatrix}{}
X^{1}_{t_{0}}\\
X^{1}_{t_{1}}\\
\vdots \\
X^{1}_{t_{n}}
\end{pmatrix},
\begin{pmatrix}{}
X^{2}_{t_{0}}\\
X^{2}_{t_{1}}\\
\vdots \\
X^{2}_{t_{n}}
\end{pmatrix},
\dotsb 
\]
is exchangeable.
\end{prop}

\begin{proof}
This is a routine consequence of hypotheses (a) and (b) of
Definition \ref{definition:exchangeableMP} and the Markov property.
\end{proof}

\begin{cor}\label{corollary:discreteProj}
For every $t\geq0$, $(X^{n}_{t})_{n\in \zz{N}}$ is an
exchangeable sequence of $[k]$-valued random variables.  Consequently, the limiting frequencies $(Y^j_t)_{j\in[k]}$ \eqref{eq:limFreqs} exist with
probability $1$. Furthermore,  the projection $(Y_{t})_{t\geq 0}$ is a
Markov process on the simplex $\simplexk$.
\end{cor}

\begin{note}\label{note:caution}
Only the almost sure existence of limiting frequencies $Y_{t}$ at fixed times
$t$ is asserted here. The almost sure existence of limiting frequencies at
\emph{all} times $t$ will be proved later. Thus, the second assertion
of the corollary, that the projection $Y_{t}$ is a Markov process,
should be interpreted only as a statement about conditional
distributions at finitely many time points.
\end{note}

\begin{proof}
The first assertion is an immediate consequence of de Finetti's
theorem. Denote by $\mathcal{F}_{t}:=\sigma(X_s)_{0\leq s\leq t}$ the $\sigma$-algebra generated
by the random variables $X_{s}$ for $s\leq t$. Since each $Y_{t}=\pi
(X_{t})$ is a measurable function of $X_{t}$, the hypothesis that $(X_{t})_{t\geq0}$ is Markov implies that the conditional distribution of
$Y_{t+r}$ given $\mathcal{F}_{t}$ coincides with the conditional
distribution of $Y_{t+r}$ given $X_{t}$. The process
$(X_{t})_{t\geq 0}$ is exchangeable; therefore, for any permutation $\sigma:\zz{N}\rightarrow\zz{N}$ fixing all but finitely many elements, the joint
distributions of $(X_{t}^{\sigma},X_{t+r}^{\sigma})$ and $(X_{t},X_{t+r})$ are identical. The projection $\pi:[k]^{\zz{N}}\rightarrow\simplexk$ is invariant
by $\sigma$, so it follows that the conditional distribution of
$Y_{t+r}$ given $X_{t}^{\sigma}$ is the same as the conditional
distribution of $Y_{t+r}$ given $X_{t}$. Thus, the conditional
distribution of $Y_{t+r}$ given $X_{t}$ depends only on the
exchangeable $\sigma$-algebra.  By de Finetti's theorem, this
is the conditional distribution of $Y_{t+r}$ given $Y_{t}$; hence,
$(Y_{t})_{t\geq 0}$ is Markov.
\end{proof}

\begin{cor}\label{corollary:stochasticMatrixProj}
For every pair of times $s<t$ and every
pair of colors $i,j\in [k]$,
\begin{equation}\label{eq:limitFreqsPairs}
	Q_{s,t} (i,j):=\lim_{n \rightarrow \infty}
	\frac{\sum_{m=1}^{n}\mathbf{1}\{X^{m}_{s}=i \quad
	\text{and}\;\; X^{m
	}_{t}=j\}}{\sum_{m=1}^{n}\mathbf{1}\{X^{m}_{s} =i\}} 
\end{equation}
exists on the event $Y^{i}_{s}>0$. With probability one, the matrix
$Q_{s,t}$ is stochastic (with the convention that, for every color $i$
such that $Y^{i}_{s}=0$, the $i$th  row $Q_{s,t} (i,\cdot)=\delta_{i}
(\cdot)$). 
\end{cor}

\begin{proof}
By Proposition~\ref{proposition:exchangeableKernel}, the vector-valued
sequence 
\[
	\begin{pmatrix}{}
X^{1}_{s}\\
X^{1}_{t}
\end{pmatrix}, 
\begin{pmatrix}{}
X^{2}_{s}\\
X^{2}_{t}
\end{pmatrix},
\dotsb 
\]
is exchangeable. Hence, by de Finetti's theorem, limiting frequencies
for all color pairs $i,j\in[k]$ exist almost surely.  Since the number of colors $k$ is finite, the corresponding matrix $Q_{s,t}$ is almost surely stochastic.
\end{proof}

\begin{prop}\label{proposition:conditional-IH}
Let $A$ be any finite set of times $0<s_{0}<s_{1}<\dotsb
<s_{n}<\infty$. Then, conditional on the exchangeable $\sigma$-algebra
$\mathcal{E}_{A}$, the vectors $X^{1}_{A},X^{2}_{A},\dotsc$ are
independent, identically distributed, and, for each $m\in \zz{N}$, the
sequence $\{X^{m}_{s_{r}} \}_{0\leq r\leq n}$ is a time-inhomogeneous
Markov chain on $[k]$ with transition probabilities
$Q_{s_{r},s_{r+1}}$ in \eqref{eq:limitFreqsPairs}. 
\end{prop}

\begin{proof}
That the vectors $X^{1}_{A},X^{2}_{A},\dotsc$ are conditionally
i.i.d.\ given $\mathcal{E}_{A}$ follows from
Proposition~\ref{proposition:exchangeableKernel} and de Finetti's
theorem.  It suffices to show that, conditional on
$\mathcal{E}_{A}$, the sequence $\{X^{m}_{s_{r}}
\}_{0\leq r\leq n}$ is Markov with the indicated transition
probabilities. 

By Proposition~\ref{proposition:exchangeableKernel}, the sequence of
random vectors $X^{1}_{A},X^{2}_{A},\dotsc$ is exchangeable, so de
Finetti's theorem and the strong law of large numbers imply that
limiting empirical frequencies exist for all vectors $j_{A}\in
[k]^{A}$. We must show that these limiting empirical frequencies are
given by
\[
	\lim_{n \rightarrow \infty}n^{-1}
	\sum_{r=1}^{n}\mathbf{1}\{X^{r}_{A}=j_{A} \} 
	=P (X^{1}_{s_{0}}=j_{0}\,|\,\mathcal{E}_{s_{0}}) 
	\prod_{l=1}^{m} Q_{s_{l-1},s_{l}} (j_{l-1},j_{l});
\]
that is, we must show that the limiting empirical frequency of the
color $j_{s_{n}}$ at those sites taking color $j_{s_{n-1}}$ at time
$s_{n-1}$ does not depend on the colors $j_{s_{r}}$ taken by these
sites at earlier times $s_{r}$. However, this is merely a consequence
of the Markov property of $X_{s_0},X_{s_1},\ldots$, as we now argue.

Note that, unconditionally, we can generate
$X_{s_0},X_{s_1},\ldots,X_{s_n}$ as follows.  For each $1\leq r\leq
n-1$, the law of $X_{s_{r+1}}$, given $\sigma(X_{s_j})_{0\leq j\leq
r}$, is a measurable function of $\pi(X_{s_r})=:Y_{s_r}$.  Moreover,
by Corollary \ref{corollary:stochasticMatrixProj}, the matrix
$Q_{s_r,s_{r+1}}$ is a random stochastic matrix with distribution
depending only on $\mathcal{E}_{s_r}$, where $Q_{s_r,s_{r+1}}(i,j)$
gives the limiting frequency of coordinates for which $X_{s_{r+1}}=j$
among those for which $X_{s_r}=i$.  Hence, given
$\sigma(X_{s_j})_{0\leq j\leq r}$, we generate a matrix
$Q_{s_r,s_{r+1}}$ randomly and, conditional on $Q_{s_r,s_{r+1}}$, we
generate $X_{s_{r+1}}$ coordinate-by-coordinate by determining
$X_{s_{r+1}}^l$, for each $l\in\mathbb{N}$, independently from
\begin{equation}\label{eq:IH-tps}
	 \mathbb{P}\{X_{s_{r+1}}^l=j|Q,X_{s_r}\}=Q(X_{s_r}^l,j).
\end{equation}
Clearly, by the strong law of large numbers, the matrix
$Q_{s_r,s_{r+1}}$, from \eqref{eq:limitFreqsPairs}, is identical to
$Q$ for those $i$ such that $Y^i_{s_r}>0$.  Moreover,
$Q_{s_r,s_{r+1}}$ is measurable with respect to
$\mathcal{E}_{(X_{s_r},X_{s_{r+1}})}\subset\mathcal{E}_A$.  By this
and \eqref{eq:IH-tps}, we conclude that $\{X_{s_r}^m\}_{0\leq r\leq
n}$ is, conditionally on $\mathcal{E}_A$, a time-inhomogeneous Markov
chain on $[k]$ with transition probabilities
\eqref{eq:limitFreqsPairs}.

\end{proof}

We now give a complete description of the set of
discrete-time exchangeable Markov chains $(X_{t})_{t\in \zz{Z}_{+}}$
on $[k]^{\zz{N}}$. Recall that $\mathcal{S}_{k}$ denotes the space of
stochastic $k\times k$ matrices, and denote by $\mathcal{P} (S_{k})$
the set of all Borel probability measures on $\mathcal{S}_{k}$. 

\begin{thm}\label{theorem:discrete-time-char}
Let $(X_{t})_{t\in \zz{Z}_{+}}$ be a discrete-time exchangeable Markov
chain on $[k]^{\zz{N}}$, and denote by $Y_{t}=\pi (X_{t})$ the
limiting empirical frequency (row) vector at time $t$ and by
$Q_{n+1}=Q_{n,n+1}$ the limiting empirical transition matrix defined
by \eqref{eq:limitFreqsPairs}. Then there exists a measurable mapping
$G:\simplexk\rightarrow \mathcal{P} (\mathcal{S}_{k})$ such that 
\begin{gather}\label{eq:Q-Law}
	\mathcal{D} (Q_{n+1}\,|\,\sigma (Y_{t},Q_{t})_{0\leq t\leq
	n})=G (Y_{n}) \quad \text{and} \\
\label{eq:Y-by-Q}
	Y_{n+1}=Y_{n}Q_{n+1}.
\end{gather}
Moreover, for each $n$, conditional on $\sigma (X_{t})_{t\leq n}$, the
coordinate variables $X^{m}_{n+1}$ are conditionally independent with
marginal distributions
\begin{equation}\label{eq:X-cond}
	P (X^{m}_{n+1}=j\,|\, X_{n})= Q_{n+1} (i,j) \quad \text{on} \;\; X^{m}_{n}=i.
\end{equation}
Conversely, for any measurable mapping $G:\simplexk\rightarrow
\mathcal{P} (\mathcal{S}_{k})$ and any initial point $Y_{0}\in
\simplexk$, there is a unique Markov chain $(X_{n})_{n\geq 0}$ on
$[k]^{\zz{N}}$ specified by equations \eqref{eq:Q-Law},
\eqref{eq:Y-by-Q},  \eqref{eq:X-cond}, and the initial condition that
the random variables $X^{m}_{0}$ are i.i.d.\ with common distribution $Y_{0}$.
\end{thm}
\begin{proof}
Let $(X_t)_{t\in\zz{Z}_+}$ be an exchangeable Markov chain on $[k]^{\zz{N}}$, then, by (b) of Definition \ref{definition:exchangeableMP}, the law of $X_{n+1}$ given $\sigma(X_t)_{0\leq t\leq n}$ depends only on $Y_{n}=\pi(X_{n})$, which exists almost surely by Corollary \ref{corollary:stochasticMatrixProj}.  Also implied by Corollary \ref{corollary:stochasticMatrixProj} is that the conditional law of $Q_{n+1}$ given $\sigma(Y_t,Q_t)_{0\leq t\leq n}$ depends only on $Y_n$, which implies \eqref{eq:Q-Law}.  Proposition \ref{proposition:conditional-IH} now implies \eqref{eq:Y-by-Q} and \eqref{eq:X-cond}.

The converse is immediate from the following construction.  To begin, we choose $X_0$, given $Y_0$, by sampling its coordinates $X_0^1X_0^2\cdots$ i.i.d.\ from $Y_0$.  Subsequently, for each $n\geq1$, given $X_n$ has limiting frequency $Y_n=\pi(X_n)$, we draw a random stochastic matrix $S$ from $G(Y_n)$, which acts as the transition probability matrix for each coordinate, as in Proposition \ref{proposition:conditional-IH}.
\end{proof}

\section{The Color-Swatch Process}\label{section:color-swatch-process}
Assume now that $(X_{t})_{t\geq 0}$ is an exchangeable,
continuous-time Markov process on $[k]^{\zz{N}}$. By
Proposition~\ref{proposition:exchangeableKernel}, for any finite set
$A\subset \zz{R}_{+}$, the vector-valued sequence
$X^{1}_{A},X^{2}_{A},\dotsc$ is exchangeable. Consequently, by the
Hewitt-Savage theorem, for any fixed $t\geq 0$, the sequence
$X^{1}_{t},X^{2}_{t},\dotsc$ has limiting empirical frequencies with
probability one. Since $\zz{R}_{+}$ is uncountable, it does not directly
follow that limiting empirical frequencies exist simultaneously for
all $t\geq 0$. In this section, we prove that, under the additional
hypothesis that the sample paths are cadlag, limiting empirical
frequencies exist for all times $t$, and that the induced
projection of $X_{t}$ to the $(k-1)$-dimensional simplex is a Markov process with
cadlag sample paths.

\subsection{Color Swatches}\label{ssec:colorSwatches} Viewing
$[k]$ as a collection of $k$ distinct colors, we may regard any cadlag
path $\{x (t) \}_{t\in [0,1]}$ in $[k]$ as a \emph{color swatch}, that
is, a concatenation of finitely many nonoverlapping colored intervals
$J_{i}$ whose union is $[0,1]$. In general, a color swatch is
parametrized by
\begin{enumerate}
\item [(i)] a positive integer $m+1$
(the number of intervals $J_{i}$); 
\item [(ii)] a sequence $0=t_{0}<t_{1}<\dotsb <t_{m+1}=1$,  the
endpoints of the intervals $J_{i}$; and
\item [(iii)] a list
$(\kappa_{1},\kappa_{2},\dotsc ,\kappa_{m+1})$  of colors
(elements of $[k]$) subject to the restriction that no two successive
colors in the list are the same.
\end{enumerate}
The space $\mathcal{W}$ of all color swatches is naturally partitioned
as $\mathcal{W}=\bigcup_{m=0}^{\infty}\mathcal{W}_{m}$, where
$\mathcal{W}_{m}$ is the set of all color swatches with $m+1$
distinctly colored sub-intervals. Furthermore, $\mathcal{W}$ is
equipped with the \emph{Skorohod metric} $d$, which is defined
as follows: for any pair $f,g\in \mathcal{W}$,
\begin{equation}\label{eq:skorohod}
	d (f,g)= \inf_{h} \max (\xnorm{h-id}_{\infty}, \xnorm{f-g\circ h}_{\infty}).
\end{equation}
Here $\xnorm{\cdot}_{\infty}$ denotes the sup norm, and the infimum is
over all increasing homeomorphisms $h:[0,1]\rightarrow [0,1]$.  The
elements $f,g\in \mathcal{W}$ are viewed as functions on [0,1] with
range $[k]$, which inherits the Euclidean norm on $\zz{R}$.  The
topology on $\mathcal{W}$ induced by the Skorohod metric is separable
\cite{billingsley}, and therefore $\mathcal{W}$ is a Borel space.
Thus, in particular, the Hewitt-Savage extension of the de Finetti
theorem \cite{hewitt-savage} applies to exchangeable sequences of
$\mathcal{W}$-valued random variables.

\begin{prop}\label{proposition:colorSwatchSequence}
If $(X_{t})_{t\in \zz{R}_{+}}$ is an exchangeable Markov process on
$[k]^{\zz{N}}$ with cadlag sample paths, then the sequence 
\begin{equation}\label{eq:colorSwatchSeq}
	X^{1}_{[0,1]}, X^{2}_{[0,1]},\dotsc 
\end{equation}
of color swatches induced by the coordinate processes $X^{i}_{t}$ is
an exchangeable sequence of $\mathcal{W}$-valued random variables.
Consequently,  with probability
one the empirical distributions $m^{-1}\sum_{i=1}^{m}
\delta_{X^{i}_{[0,1]}}$ converge weakly to a random Borel probability
measure $\Theta$ on $\mathcal{W}$. The random measure $\Theta$ generates the
exchangeable $\sigma$-algebra $\mathcal{E}_{[0,1]}$ of the sequence
\eqref{eq:colorSwatchSeq}, and, conditional on $\mathcal{E}_{[0,1]}$,
the color swatches $X^{1}_{[0,1]}, X^{2}_{[0,1]},\dotsc$ are
independent, identically distributed with distribution $\Theta$.
\end{prop}

\begin{proof}
By assumption, each $X^i_{[0,1]}$, $i=1,2,\ldots$, is cadlag, and so
$X^i_{[0,1]}\in\mathcal{W}$ for every $i\in\mathbb{N}$ almost surely.
That the sequence \eqref{eq:colorSwatchSeq} is exchangeable follows
from Proposition~\ref{proposition:exchangeableKernel} because the
$\sigma$-algebra $\sigma\langle
X^i_{\zz{Q}\cap[0,1]},\,i\in\mathbb{N}\rangle$ generated by the
coordinate variables $x_{q}$, where $q$ is rational, generates the
Borel $\sigma$-algebra on $\mathcal{W}$. The rest follows from the
Hewitt-Savage and de Finetti theorems.
\end{proof}

\begin{cor}\label{corollary:IH-description}
If $(X_{t})_{t\in \zz{R}_{+}}$ is an exchangeable Markov process on
$[k]^{\zz{N}}$ with cadlag sample paths, then, conditional on
$\mathcal{E}_{[0,1]}$,  the color swatches $X^{m}_{[0,1]}$ are i.i.d.\ and
$X^{m}_{[0,1]}$ has the same conditional law as the path of an
inhomogeneous, continuous-time Markov chain on the finite state
space $[k]$ with transition probabilities $Q_{s,t}$; that is, for all
$0\leq s\leq t\leq 1$, 
\begin{equation}\label{eq:IH-description}
	P (X^{m}_{t}=j \,|\, \mathcal{E}_{[0,1]}\vee  \sigma
	(X^{m}_{r})_{r\leq s}) = Q_{s,t} (i,j) \quad \text{on}\;\; X_{s}^m=i.
\end{equation}
\end{cor}

\begin{proof}
We have already shown, in Proposition
\ref{proposition:colorSwatchSequence}, that the color swatches
$X^{m}_{[0,1]}$ are conditionally i.i.d.\ given the exchangeable
$\sigma$-algebra $\mathcal{E}_{[0,1]}$. That the individual color
swatches are conditionally inhomogeneous Markov processes with
transition probabilities \eqref{eq:IH-description} follows routinely
from Proposition~\ref{proposition:conditional-IH}, because the
exchangeable $\sigma$-algebra $\mathcal{E}_{[0,1]}$ is generated by
$\vee_{n=1}^{\infty}\mathcal{E}_{\mathcal{D}_{n}}$, where
$\mathcal{D}_{n}$ is the set of $n$th-level dyadic rationals
$m/2^{n}$ in [0,1].
\end{proof}

\subsection{Existence of limiting empirical color frequencies}\label{ssec:empirical} 

Proposition \ref{proposition:colorSwatchSequence} implies that if
$X_{t}$ is an exchangeable Markov process on $[k]^{\zz{N}}$ with
cadlag sample paths, then the joint distribution of the sequence $(X^{i}_{[0,1]})_{i\in\zz{N}}$ of color
swatches is a mixture of product measures. Thus, to prove that this
sequence has well-defined limiting empirical color frequencies  at all
times $t\in [0,1]$, it suffices to show that this is the case for any
sequence of i.i.d.\ color swatches.

\begin{prop}\label{proposition:empiricalFrequencies}
Let $Z^{1},Z^{2},\dotsc $ be a sequence of
independent, identically distributed $\mathcal{W}$-valued random
variables. Then, with probability one, for every $t\in [0,1]$ the
empirical distributions of the sequence $(Z^{i}_{t})_{i\in\mathbb{N}}$
converge to a non-random probability distribution $\pi (Z_{t}):=(\pi(Z_t)^j)_{j\in[k]}$ on
$[k]$; that is, for each $j\in [k]$ and $t\in [0,1]$,
\begin{equation}\label{eq:empiricalFrequencies}
	\lim_{n \rightarrow \infty}n^{-1}\sum_{i=1}^{n}
	\mathbf{1}\{Z^{i}_{t}=j \} =\pi (Z_{t})^{j}.
\end{equation}
\end{prop}

\begin{proof}
To prove that the sequence of empirical frequencies converges, we need
only show that, for any $\varepsilon >0$, the upper and lower limits
of the sequence of averages \eqref{eq:empiricalFrequencies} differ by
no more than $\varepsilon$. Fix $j\in [k]$ and, for each $t\in [0,1]$,
define
\begin{align*}
	L^{+}_{t}&=\limsup_{n \rightarrow \infty}n^{-1}\sum_{i=1}^{n}
	\mathbf{1}\{Z^{i}_{t}=j \} \quad \text{and}\\
	L^{-}_{t}&=\liminf_{n \rightarrow \infty}n^{-1}\sum_{i=1}^{n}
	\mathbf{1}\{Z^{i}_{t}=j \} .
\end{align*}
For any fixed $t\in [0,1]$, the strong law of large numbers implies
that $L^{+}_{t}=L^{-}_{t}$ almost surely.
We will show that with probability one,
$L^{+}_{t}-L^{-}_{t}\leq2\varepsilon$ for all $t\in [0,1]$.

By hypothesis, the random functions $Z^{1},Z^{2},\dotsc$ have cadlag
sample paths, and in particular, each $Z^{i}$ has at most finitely
many discontinuities in the time interval $[0,1]$.  Define
\begin{equation}\label{eq:fixed-discontinuities}
	D=\{t\in [0,1]\,:\, P\{Z^{i}_{t}\not =Z^{i}_{t-} \}>0\}.
\end{equation}
We claim that this set is at most countable. For suppose not; then
for some $\varepsilon >0$, the subset $D_{\varepsilon}$ consisting of
times $t\in D$ such that $P\{Z^{i}_{t}\not =Z^{i}_{t-} \}\geq \varepsilon$
would be uncountable. But by the SLLN, for every $t\in D_{\varepsilon}$, the
limiting fraction of  paths $Z^{i}$ that are discontinuous at $t$ is
at least $\varepsilon$.  As a consequence, the limiting fraction of
paths $Z^{i}$ that have infinitely many discontinuities in [0,1] must be
positive, contradicting the hypothesis that each all paths $Z^{i}$ are
cadlag. It follows that $D$ is at most countable.  Furthermore, for each 
$\varepsilon >0$, the subset $D_{\varepsilon }\subset D$ must be finite.

Fix $\varepsilon >0$, and consider the set $R_{\varepsilon
}=[0,1]\setminus D_{\varepsilon}$. This set is the union of finitely
many (relatively) open subintervals of [0,1] whose endpoints are
elements of $D_{\varepsilon}$. For each time $t\in R_{\varepsilon}$,
the probability that $Z^{i}$ has a discontinuity at time $t$ is less
than $\varepsilon$. We claim that there is a \emph{countable}
partition $J_1,J_{2},\ldots $ of $R_{\varepsilon}$ into
non-overlapping intervals $J_{l}$ such that, for each interval
$J_{l}$,
\begin{equation}\label{eq:Ji-char}
	P\{Z^{i}  \;\text{is discontinuous at some} \; t\in J_{l}\}<\varepsilon.
\end{equation}
To construct such a partition, begin by dividing the constituent
intervals of $R_{\varepsilon}$ into halves, and then dividing these
halves into quarters, etc., and stopping the subdivision process
whenever an interval $J_{l}$ is small enough that \eqref{eq:Ji-char}
holds.  This must happen eventually for each interval \emph{not}
abutting one of the points of $D_{\varepsilon}$, because otherwise
there would  be a nested sequence of closed intervals
$J_{n}\subset R_{\varepsilon }$ that shrink to a point $t\in
R_{\varepsilon}$, and for each $n\geq 1$ 
\[
	P\{Z^{i} \;\text{has a discontinuity in}\; J_{n}\}\geq \varepsilon .
\]
But this would imply that $t\in D_{\varepsilon}$, a contradiction.

Suppose, then, that $J_{l}$ is a countable partition of $R_{\varepsilon}$
into non-overlapping intervals $J_{l}$ such that for each interval
$J_{i}$ condition \eqref{eq:Ji-char} holds. Then the SLLN implies
that  for each interval $J_{l}$ of the partition, 
\[
	\lim_{n \rightarrow \infty}n^{-1}\sum_{i=1}^{n}
	\mathbf{1}\{Z^{i}  \;\text{is discontinuous at some} \; t\in
	J_{l}\}<\varepsilon . 
\]
It follows that neither $L^{+}_{t}$ nor $L^{-}_{t}$ can vary by more
than $\varepsilon$ over the interval $J_{i}$.  Since
$L^{+}_{s}=L^{-}_{s}$ almost surely for each endpoint $s$ of $J_{i}$
(this set is countable!), we conclude that for each interval $J_{i}$, with
probability one,
\[
	\sup_{t\in J_{i}} L^{+}_{t}-L^{-}_{t} \leq 2\varepsilon .
\]
Since $[0,1]$ is covered by the intervals $J_{i}$ and the non-random
finite set $D_{\varepsilon }$ (on which $L^{+}_{t}=L^{-}_{t}$ a.s.),
it follows that, for every $\varepsilon>0$,
\[
	 \sup_{t\in[0,1]}L_t^+-L_t^-\leq2\varepsilon\quad\text{a.s.}
\]
Since $\varepsilon >0$ is arbitrary, the limits
\eqref{eq:empiricalFrequencies} exist for all $t\in [0,1]$ almost
surely.  Since, for any finite $n\in\mathbb{N}$, the limits
\eqref{eq:empiricalFrequencies} do not depend on the first $n$
elements of $(Z^i_{[0,1]})_{i\in\mathbb{N}}$, Kolmogorov's 0-1 law
implies the limits are deterministic.

\end{proof}

\begin{cor}\label{corollary:cadlagProjections}
If $(X_{t})_{t\geq 0}$ is an exchangeable, continuous-time Markov
process on $[k]^{\zz{N}}$ with cadlag sample paths then, for every
$t\geq 0$, the empirical limiting frequency vector $Y_{t}=\pi (X_{t})$
exists, and the projection $(Y_{t})_{t\geq 0}$ is a Markov process
with cadlag sample paths in the simplex $\simplexk$.
\end{cor}

\begin{proof}
The existence of limiting empirical frequencies for all $t\geq 0$
follows from Propositions~\ref{proposition:colorSwatchSequence} and
\ref{proposition:empiricalFrequencies} and the Hewitt-Savage extension
of de Finetti's theorem \cite{hewitt-savage}. The Markov property of $Y_{t}$ follows from
Corollary~\ref{corollary:discreteProj}. Finally, the sample paths of
$(Y_{t})_{t\geq 0}$ are cadlag since, by Proposition
\ref{proposition:empiricalFrequencies}, $(Y_t)_{t\geq0}$ is continuous
at every $t\notin D$ almost surely, where $D$ is defined by
\eqref{eq:fixed-discontinuities}, and $D$ is a countable, non-random
subset of $[0,1]$.
\end{proof}

\subsection{Characterization of
discontinuities}\label{ssec:discontinuities}

\begin{prop}\label{proposition:discontinuities}
Let $(X_{t})_{t\in \zz{R}_{+}}$ be an exchangeable Markov process on
$[k]^{\zz{N}}$ with cadlag sample paths. Then, with probability one, the
sample path $t\mapsto X_{t}$ has at most countably many
discontinuities. Furthermore, with probability one, there are precisely two possible types
of discontinuity at $s$: either 
\begin{enumerate}
\item [(I)] there exists a unique $i\in \zz{N}$ such that $X^{i}_{t}$
is discontinuous at $s$, or
\item [(II)] $P ( X^{1}_{t}  \;\text{is discontinuous at } \; t=s \,|\, \mathcal{E})>0$,
and for each pair $j_{1}\not =j_{2}\in [k]$,
\begin{equation}\label{eq:limFreqJump}
	\lim_{n \rightarrow \infty} n^{-1}\sum_{i=1}^{n} 
	\mathbf{1}\{X^{i}_{s-}=j_{1} \;\text{and}\; X^{i}_{s}=j_{2}\}
	= P (X^{1}_{s-}=j_{1} \;\text{and}\; X^{1}_{s}=j_{2} \,|\, \mathcal{E}).
\end{equation}
\end{enumerate}
Moreover, the projection $Y_{t}=\pi (X_{t})$ to the simplex has cadlag
paths, with discontinuities only at the times of type-(II)
discontinuities of $(X_{t})_{t\in\mathbb{R}^+}$.
\end{prop}

\begin{proof}
It suffices to prove the corresponding assertions for the restriction
$X_{[0,1]}$, because the case $X_{[0,T]}$ for arbitrary $T$ follows by
rescaling, and the case $X_{[0,\infty )}$ follows by exhaustion.
By Proposition~\ref{proposition:colorSwatchSequence}, it suffices to
prove the corresponding statement for sequences of i.i.d.\ color
swatches. Thus, we assume that $Z^{1},Z^{2},\dotsc$ are independent,
identically distributed $\mathcal{W}$-valued random variables and, as
in the proof of Proposition \ref{proposition:empiricalFrequencies}, we
define $D_{\varepsilon }$ to be the set of all $s\in [0,1]$ such that
$P\{Z^{1}_{t} \; \text{discontinuous at}\; t=s \}\geq \varepsilon $
and $D:=\bigcup_{\varepsilon >0}D_{\varepsilon }$.  As was shown in
the proof of Proposition~\ref{proposition:empiricalFrequencies},  
$D$ is at most countable.  By the strong law
of large numbers, for each $s\in D$ and any pair $j_{1}\not =j_{2} \in
[k]$,
\[
	\lim_{n \rightarrow
	\infty}n^{-1}\sum_{i=1}^{n}\mathbf{1}\{Z^{i}_{s-}=j_{1}
	\;\text{and } Z^{i}_{s}=j_{2} \}=P\{Z^{1}_{s-}=j_{1}
	\;\text{and } Z^{1}_{s}=j_{2} \}.
\]

Next, we show that for $s\not \in D$ there can be at most one index
$i\in \zz{N}$ for which $Z^{i}_{t}$ has a discontinuity at $t=s$.  For
any $i\in\zz{N}$, let $D^i$ be the set of discontinuities of $Z^i$.
Because $(Z^i)_{i\in\zz{N}}$ is an i.i.d.\ collection of color
swatches,
\begin{equation*}
P\left\{Z^{j}\text{ discontinuous at some} \;t\in D^i\setminus
D\,\left|\,D^i\right\}\right.=0,\quad\text{for all }j\neq i,
\end{equation*}
since for any $t\not \in D$ the \emph{unconditional} probability that $Z^{j}$ is
discontinuous at $t$ is $0$. 
Hence, unconditionally, the probability that  $Z^i$ and
$Z^{j}$ share a common point of discontinuity outside $D$ is 0. 
Since there are only countably many pairs $i,j$, it follows that
there is zero probability that some pair has a common discontinuity
outside of $D$.

Finally, by the proof of Proposition
\ref{proposition:empiricalFrequencies}, for each $\varepsilon >0$
there is a countable partition of the open set $[0,1]\setminus
D_{\varepsilon}$ into intervals $J_{l}$ such that \eqref{eq:Ji-char}
holds. It follows by the SLLN that the limiting empirical frequencies
$\pi(Z_t)$ cannot vary by more than $\varepsilon$ in any of the
intervals $J_{l}$. Since $\varepsilon >0$ can be made arbitrarily
small, it follows that the only discontinuities of the projection $\pi
(Z_{t})$ must be at points $t\in D$.

\end{proof}

\section{Bounded Variation of Sample Paths in $\simplexk$}\label{sec:bv}

\subsection{Mass Transfer}\label{ssec:mass-transfer} Let
$(X_{t})_{t\geq 0}$ be a continuous-time, exchangeable Markov process
on $[k]^{\zz{N}}$ with cadlag sample paths, and let $Y_{t}=\pi
(X_{t})$ be its projection to the simplex.
Corollary~\ref{corollary:stochasticMatrixProj} implies that for every
pair of colors $i,j\in [k]$ and any two times $s\leq t$, the limiting
fraction $Q_{s,t} (i,j)$ of sites with color $i$ at time $s$ that flip
to color $j$ at time $t$ exists, and the matrix $Q_{s,t}$ is
stochastic. The results of Section~\ref{section:color-swatch-process}
imply that the process $Q_{s,t}$ can be extended to a two-parameter
process $\{Q_{s,t} \}_{0\leq s\leq t}$ valued in the space
$\mathcal{S}_{k}$ of $k\times k$ stochastic matrices. For each fixed
$s$, the sample paths $\{Q_{s,t} \}_{s\leq t}$ are cadlag, with
discontinuities only at the type-(II) discontinuities of
$X_{t}$. Furthermore, the matrices $Q_{s,t}$ satisfy the \emph{cocycle
equations} \eqref{eq:cocycle} and the compatibility condition \eqref{eq:q-y}.
We define the \emph{total mass transfer} $T_{s,t}$ between times $s$ and
$t$ to be the limiting fraction of sites $m\in \zz{N}$ that have
different colors at times $s$ and $t$; that is,
\begin{equation}\label{eq:mass-transfer}
	T_{s,t}:=\sum_{a\not =b} Y^{a}_{s} Q_{s,t} (a,b)
	=\lim_{n \rightarrow \infty}n^{-1}\sum_{i=1}^{n}
	\mathbf{1}\{X^{i}_{s}\not = X^{i}_{t} \} .
\end{equation}
Observe that the variation
$\xnorm{Y_{t}-Y_{s}}_{1}:=\sum_{i=1}^k|Y_t^i-Y_s^i|$ in the frequency
vector between times $s$ and $t$ is bounded above by
$T_{s,t}$. Moreover, the  mass transfer process $(T_{s,t})_{0\leq
s\leq t}$ satisfies
\begin{equation}\label{eq:sub-cocycle}
	T_{r,t}\leq T_{r,s}+T_{s,t} \quad \text{for all} \;\; r\leq s\leq
	t.
\end{equation}

\begin{prop}\label{proposition:bv}
With probability one, the mass transfer process $(T_{s,t})_{0\leq s\leq t}$  
has bounded variation on finite intervals, that is,
\begin{equation}\label{eq:mass-transfer-bv} 
	\sup_{0=s_{0}\leq s_{1}\leq s_{2}\leq \dotsb \leq s_{n}=1}
	\sum_{i=0}^{n-1}T_{s_{i},s_{i+1}} <\infty  \quad \text{a.s.}
\end{equation}
\end{prop}

\subsection{Proof of Proposition
\ref{proposition:bv}}\label{ssec:binaryCase-bv} We first prove
Proposition \ref{proposition:bv} in the special case $k=2$.  This case
is more straightforward than the general case because if the mass
transfer process has unbounded variation, then any transfer of mass
from 0 to 1 must eventually be matched by a corresponding transfer
from 1 to 0; whereas, in the general case, mass can be transported along
incomplete cycles in the complete graph $K_{[k]}$.

\begin{proof}[Proof for the case $k=2$]

Assume that \eqref{eq:mass-transfer-bv} does not hold; then there is a
positive probability that the supremum in \eqref{eq:mass-transfer-bv} is
$+\infty$.  Note that this event is in the exchangeable $\sigma$-algebra 
$\mathcal{E}_{[0,1]}$.  Furthermore, on this event, for every $M\geq1$,
we can partition $[0,1]$ into non-overlapping sub-intervals $J_1,\ldots,J_M$
such that the total variation of the mass transfer process within each sub-interval
exceeds $M$.  The mass transfer from color $a$ to color
$b$  in any time  interval $(s,t]$ is bounded above by $Q_{s,t}
(a,b)$; and so, to each $J_{i}$, $i=1,\ldots,M$, there must be a further partition into $N$ non-overlapping subintervals $J_{ij}= (s_{ij},t_{ij})$
such that
\begin{equation}\label{eq:q-variation}
	\sum_{j=1}^{N} Q_{s_{ij},t_{ij}} (0,1) \geq M
	\quad \text{and} \quad 
	\sum_{j=1}^{N} Q_{s_{ij},t_{ij}} (1,0) \geq M.
\end{equation}

Recall that, by Corollary~\ref{corollary:IH-description}, the coordinate
process $(X^{m}_{t})_{t\in [0,1]}$ is, conditional on the exchangeable
$\sigma$-algebra $\mathcal{E}_{[0,1]}$, an inhomogeneous
continuous-time Markov chain on $[k]$ with transition probabilities
$Q_{s,t}$. Consequently, on the event \eqref{eq:q-variation}, the
conditional probability that $X^{1}_{t}$ has no discontinuity in $J_{i}$ is bounded above by
\[
	\max \left(\prod_{j=1}^{N}  (1-Q_{s_{ij},t_{ij}} (0,1)),
	    \prod_{j=1}^{N}  (1-Q_{s_{ij},t_{ij}} (1,0)) \right)
	    \leq e^{-M}.
\]
 Hence,
\[
	P (X^{1}_{[0,1]} \;\text{has at least }\; M \; \text{discontinuities}
	\,|\, \mathcal{E}_{[0,1]}) \geq 1-Me^{-M},
\]
on the event that there exist non-overlapping intervals $J_{ij}$
satisfying \eqref{eq:q-variation}. Since $M$ is arbitrary, it follows
that $X^{1}_{[0,1]}$ must have infinitely many
discontinuities, contradicting the hypothesis that the sample paths
$(X_{t})_{t\geq 0}$ are cadlag.  This completes the proof for $k=2$.
\end{proof}

The proof of Proposition \ref{proposition:bv} in the general case $k\geq2$ is a consequence of the same ideas as the $k=2$ case above.  The key distinction is that mass can be transferred along incomplete cycles of the complete graph $K_{[k]}$ and these cycles can vary measurably with the exchangeable $\sigma$-algebra of the process.  However, since there are only finitely many states, unbounded variation demands that, for every $M\geq1$, there must be at least one starting state $j\in[k]$ for which $Y_0^j>0$ and the conditional probability that $X^1_{[0,1]}$ has at least $M$ discontinuities given $X_0^1=j$ is at least $1-Me^{-M}$.

\begin{proof}[Proof of Proposition \ref{proposition:bv} for $k\geq2$]
Again, we suppose that there is positive probability that the supremum in \eqref{eq:mass-transfer-bv} is $+\infty$.  Then, for every positive integer $M\geq1$, the total variation must exceed $M^2k^{3}$.  Therefore, there is some sequence of times $0=s_0<s_1<\cdots<s_{N}=1$ such that $\sum_{i=0}^{N-1}T_{s_i,s_{i+1}}\geq M^2 k^{3}$.  It follows that we can specify a collection $J_1,\ldots,J_{Mk}$ of non-overlapping subintervals so that the total variation within each subinterval exceeds $Mk^2$.  In this way, for each $i=1,\ldots,Mk$, there is a further subpartition $J_{i1},\ldots,J_{iN}$ of $J_i$ and a color $\kappa_i^*\in[k]$ such that 
\begin{equation}\label{eq:M-lb}\sum_{j=1}^N Y_{s_{ij}}^{\kappa_i^*}Q_{s_{ij},t_{ij}}(\kappa_i^*,[k]\setminus\{\kappa_i^*\})\geq Mk.\end{equation}
Furthermore, we have
\[Y_{s_{ij}}^{\kappa_i^*}:=\sum_{l=1}^kY_0^lQ_{0,s_{ij}}(l,\kappa_i^*),\]
and so there must be at least one $\kappa_i^0\in[k]$ for which
\begin{equation}\label{eq:M-lb-t0}\sum_{j=1}^N Y_{0}^{\kappa_i^0}Q_{0,s_{ij}}(\kappa_i^0,\kappa_i^*)Q_{s_{ij},t_{ij}}(\kappa_i^*,[k]\setminus\{\kappa_i^*\})\geq M.\end{equation}
Since $k<\infty$, any such $\kappa_i^0$ must have
$Y_0^{\kappa_i^0}>0$.  Consequently, within each subinterval, the
conditional probability that $X_{[0,1]}^1$ experiences no discontinuities in
$J_i$, given $X_{0}^1=\kappa_i^0$ and the exchangeable
$\sigma$-algebra, satisfies
\begin{align*}
	P\{X_{[0,1]}^1 &\text{ has no  discontinuity in }J_i\,|\,\mathcal{E}_{J_i},X_{0}^1=\kappa_i^0\}\\
	&\leq
	\prod_{j=1}^N\left(1-Q_{0,s_{ij}}(\kappa_i^0,\kappa_i^*)Q_{s_{ij},t_{ij}}(\kappa_i^*,[k]\setminus\{\kappa^*_{i}\})\right) \\
	&\leq e^{-M}.
\end{align*}

By the pigeonhole principle, there must be some $r\in[k]$ such that
$r=\kappa^0_i$ for at least $M$ indices $i=1,\ldots,Mk$.  In this
case, let $1\leq i_1<\cdots<i_{M}\leq Mk$ be a subset of those indices
for which $r=\kappa^0_i$.  Then we have
\begin{eqnarray*}
\lefteqn{P\{X_{[0,1]}^1\text{ has less than }M\text{ discontinuities}\,|\,\mathcal{E}_{[0,1]},\,X_0^1=r\}\leq}\\
&\leq&P\left\{\bigcup_{j=1}^M\{X_{[0,1]}^1\text{ has no discontinuity in }J_{i_j}\}\,|\,\mathcal{E}_{[0,1]},\,X_0^1=r\right\}\\
&\leq & \sum_{j=1}^MP\left\{X_{[0,1]}^1\text{ has no discontinuity in }J_{i_j}\,|\,\mathcal{E}_{[0,1]},\,X_0^1=r\right\}\\
&\leq & \sum_{j=1}^Me^{-M}\\
&=&Me^{-M}.
\end{eqnarray*}
Therefore,
\[P\left\{X_{[0,1]}^1\text{ has at least }M\text{ discontinuities}\,|\,\mathcal{E}_{[0,1]},\,X_0^1=r\right\}\geq1-Me^{-M},\]
for every $M\geq1$.  Since $Y_0^r$ must be strictly positive for any
such color, we conclude that there is positive probability that
$X_t^1$ has more than $M$ discontinuities in $[0,1]$, for every
$M\geq1$, which contradicts the assumption that $X_t^1$ has cadlag
sample paths.  We conclude that the projection $Y_t=\pi(X_t)$ must
have locally bounded variation almost surely.
\end{proof}

\section{Construction of associated chain on
$[k]^{\zz{N}}$}\label{section:construction} In this section we prove
Theorem \ref{thm:k2}, which states that, for every Markov process
$(Y_{t})_{t\geq 0}$ on the $k$-simplex $\simplexk$ whose sample paths
are cadlag and of locally bounded variation, there exists an
exchangeable Markov process $(X_{t})_{t\geq 0}$ on $[k]^{\zz{N}}$
whose projection \eqref{eq:limFreqs} into $\simplexk$ has the same
distribution as $(Y_{t})_{t\geq 0}$. The strategy is as
follows. First, we will show that the process $Y_{t}$ uniquely
determines a two-parameter process $( Q_{s,t})_{0\leq s\leq t}$ taking
values in the space of $k\times k$ stochastic matrices such that $(
Q_{s,t})_{0\leq s\leq t}$ satisfies the cocycle relations
\eqref{eq:cocycle}, the compatibility condition \eqref{eq:q-y}, and a
minimality condition \eqref{eq:minimality} spelled out below. Then,
given the two-parameter process $(Q_{s,t})_{0\leq s\leq t}$, we will
construct an i.i.d.  sequence of   time-inhomogeneous Markov chains
$X^{i}_{t}$ with transition probability matrices $Q_{s,t}$. Finally,
we will show that $[k]^{\zz{N}}$-valued process
\[
	X_{t}=X^{1}_{t}X^{2}_{t}\dotsb 
\]
is Markov, exchangeable, has cadlag paths and projects to $Y_{t}$.

\begin{prop}\label{proposition:compatibility}
If $(Y_{t})_{t\in [0,a]}$ is a cadlag path of bounded variation in
$\simplexk$ then there exists a two-parameter family $\{Q_{s,t}
\}_{0\leq s\leq t\leq a}$ of $k\times k$ stochastic matrices
satisfying \eqref{eq:cocycle} and \eqref{eq:q-y} such that for every
subinterval $[c,d]\subseteq [0,a]$,
\begin{equation}\label{eq:minimality}
		\sup_{c=s_{0}\leq s_{1}\leq s_{2}\leq \dotsb \leq
		s_{n}=d}\sum_{i=0}^{n-1} T_{s_{i},s_{i+1}} =
		\xnorm{(Y_{t})_{t\in [c,d]}}_{TV} 
\end{equation}
where $\xnorm{\cdot}_{TV}$ denotes the total variation norm on
$\simplexk$-valued paths and $T_{s,t}$ is the mass-transfer function
defined by  \eqref{eq:mass-transfer}. The mapping that sends the path
$(Y_{t})_{t\in [0,a]}$ to the two-parameter function $\{Q_{s,t} \}_{0\leq
s\leq t\leq a}$ is measurable, and for each pair $c<d$ the section $\{Q_{s,t}
\}_{c\leq s\leq t\leq d}$ depends only on $(Y_{t})_{t\in [c,d]}$.
\end{prop}

\begin{proof}
First, we consider the special case where the path $t\mapsto Y_{t}$ is
smooth, and then we will indicate the modifications necessary for the
more general case where the path is not necessarily smooth but is
continuous and has (locally) bounded variation. Finally, we will
indicate how to augment the construction to account for the
possibility of jump discontinuities. Assume that $(Y_{t})_{t\in
[0,a]}$ is smooth, and set $DY_{t}=dY_{t}/dt$.  Because each $Y_{s}$
is a probability distribution on the set $[k]$, its entries must sum
to $1$, and so the entries of the vector $DY_{s}$ must sum to $0$ for
almost every $s$. Let $J_{+}, J_{-}$ and $J_{0}$ be the sets of
indices $i\in [k]$ such that $DY^{i}_{s}>0$, $DY^{i}_{s}<0$, and
$DY^{i}_{s}=0$, respectively, and define for each $s\in [0,a]$ a
$k\times k$ rate matrix $R_{s}$ as follows:
\begin{align*}
	R_{s} (i,j)&= +DY^{i}_{s} DY^{j}_{s}/\sum_{i'\in
	J_{-}}DY^{i'}_{s} \quad \text{if} \; i\in J_{-} \; \text{and}\; j\in J_{+};\\
	R_{s} (i,j)&= -DY^{i}_{s} DY^{j}_{s}/\sum_{i'\in
	J_{-}}DY^{i'}_{s} \quad \text{if} \; i\in J_{+} \; \text{and}\; j\in J_{-};\\
	R_{s} (i,i)&= DY^{i}_{s}; \quad \text{and}\\
	R_{s} (i,j)&=0 \quad \text{otherwise}.
\end{align*}
This matrix determines the instantaneous rates of mass flow between
indices $i,j\in [k]$ following the \textsc{Marx-Engels} 
protocol (to each according to his needs; from each according to his
abilities.)  The two-parameter stochastic matrix-valued process
$Q_{s,t}$ is then defined by the matrix exponential
\begin{equation}\label{eq:marx}
	Q_{s,t}:=\exp \left\{\int_{s}^{t} R_{u} \,du \right\}.
\end{equation}
That this is in fact a stochastic matrix follows because the row sums
of each $R_{u}$ are $0$, and the compatibility condition
\eqref{eq:q-y} follows by a routine differentiation of
\eqref{eq:marx}.  The total variation identity \eqref{eq:minimality}
follows from the fact that for any smooth function $f:[a,b]
\rightarrow \zz{R}$ the total variation of $f$ on the interval $[a,b]$ is
\[
	\xnorm{f}_{TV}=\int_{a}^{b}|Df (s)|\,ds.
\]

Next, consider the more general case where $Y_{t}$ is continuous and
of locally bounded variation, but not necessarily smooth. The
preceding construction may fail because, although $t\mapsto Y_{t}$ is
differentiable at almost every $t$ (by the Lebesgue differentiation
theorem), $Y_{t}-Y_{s}$ is not always equal to the integral of the
derivative $DY_{r}$ over the interval $r\in [s,t]$, and this is needed
for the compatibility condition \eqref{eq:q-y}. To circumvent this
difficulty, we differentiate not with respect to $t$ but instead with
respect to the total variation of the path. In particular, let 
\[
	T_{t}=\xnorm{(Y_{s})_{s\in [0,t]}}_{TV}
\]
be the total variation of the path $Y$ on the time interval $[0,t]$,
and set
\[
	DY_{t}=\frac{dY_{t}}{dT_{t}};
\]
that is, $DY^{i}_{t}$ is the Radon-Nikodym derivative of the
signed measure $[a,b]\mapsto Y^{i}_{b}-Y^{i}_{a}$ with respect to the
positive measure $[a,b]\mapsto T_{b}-T_{a}$. Since the signed measure
is absolutely continuous with respect to the positive measure, it
follows that for every $t<\infty$,
\[
	Y_{t}-Y_{0}=\int_{0}^{t} DY_{s} \,dT_{s}.
\]
Now  define the rate
matrices $R_{s}$ as above and set 
\[
	Q_{s,t}:=\exp \left\{\int_{s}^{t} R_{u} \,dT_{u} \right\}.
\]
The rest of the argument now proceeds as above. Observe that this
is really a time-change argument and could be reformulated as
such. In particular,
if one defines $\tilde{Y}_{T_{t}}=Y_{t}$, then the process
$\tilde{Y}_{s}$ has differentiable paths with bounded derivative, and
so the argument for the differentiable case applies, producing a
two-parameter process $(\tilde{Q}_{s,t})_{s\leq t}$.  Reversing the
time change then gives the desired process $(Q_{s,t})_{s\leq t}$.

Finally, suppose that $t\mapsto Y_{t}$ has jump discontinuities at
times $r\in C$, where $C$ is a countable set. Because $(Y_{t})_{t\geq
0}$ has locally bounded variation, the jump sizes (measured by the
$L^{1}$-norm on the simplex $\simplexk$) are the sizes $\Delta T_{t}$
of the jumps of the total variation function $T_{t}$, which are
summable over any bounded time interval. Define a new path
$\tilde{Y}_{t}$ by ``opening'' each jump discontinuity of $Y_{t}$,
that is, at each $r\in C$ insert an interval of length $\Delta T_{r}$,
and let $\tilde{Y}_{t}$ follow the straight-line path from $Y_{r-}$ to
$Y_{r}$ (at speed $1$) during this inserted interval. The new path
$\tilde{Y}_{t}$ will still be of locally bounded variation, and it
will be continuous, so the construction of the preceding paragraph
will now yield a two-parameter family $\tilde{Q}_{s,t}$ satisfying the
cocycle identity \eqref{eq:cocycle} and the compatibility condition
\eqref{eq:q-y} for the path $\tilde{Y}_{t}$. By ``closing'' all of the
intervals opened in creating $\tilde{Y}_{t}$, we obtain a
two-parameter family $Q_{s,t}$ that satisfies the compatibility
condition \eqref{eq:q-y} for the path $Y_{t}$.
\end{proof}

\begin{proof}
[Proof of Theorem~\ref{thm:k2}]
Let $(Y_{t})_{t\geq 0}$ be a Markov process on the simplex $\simplexk$
whose sample paths are cadlag and of locally bounded variation. By
Proposition~\ref{proposition:compatibility}, there is a compatible
random semigroup $(Q_{s,t})_{0\leq s\leq t}$ such that, for any $c<d$,
the section  $(Q_{s,t})_{c\leq s\leq d}$ is a measurable function of
the path $(Y_{t})_{t\in [c,d]}$. Given the realization of
$(Q_{s,t})_{0\leq s\leq t}$, let $(X^{i}_{t})_{t\geq 0}$ be
i.i.d.\ copies of a time-inhomogeneous Markov chain on $[k]$ with
transition probability matrices $Q_{s,t}$ (cf. equation
\eqref{eq:IH-MC}) and initial distribution $Y_{0}$, and let 
\[
	X_{t}=X^{1}_{t}X^{2}_{t}\dotsb .
\]
This process is, by construction, exchangeable. It is also Markov
because, for any $s\geq 0$, the evolution of $X_{t}$ for $t\geq s$
depends, by construction, only on the state $X_{s}$ and the section
$(Q_{r,t})_{s\leq r\leq t}$ of the random semigroup, which in turn
depends only on $(Y_{t})_{t\geq s}$. That the projection of $X_{t}$ to
the simplex $\simplexk$ is $Y_{t}$ follows by the strong law of large
numbers and the compatibility condition \eqref{eq:q-y}: in particular,
\begin{eqnarray*}
\lim_{n\rightarrow\infty}n^{-1}\sum_{i=1}^{n}\mathbf{1}\{X_t^i=b\}&=&\lim_{n \rightarrow \infty}n^{-1}\sum_{i=1}^{n}\sum_{a=1}^k
	\mathbf{1}\{X^{i}_{0} =a \; \text{and} \;
	X^{i}_{t} =b\} \\
&=&\sum_{a=1}^{k}Y_0^aQ_{0,t}(a,b)\\
&=&Y^{b}_{t}.
\end{eqnarray*}
Finally, the fact  that the paths of $X_{t}$ are cadlag follows from
the minimality condition \eqref{eq:minimality}, because this implies
that the total mass transfer in any bounded time interval is
finite, which implies that the number of jumps in any of the
component chains $X^{i}_{t}$ in a bounded time interval is
finite.
\end{proof}

\section{The Feller Case}\label{section:feller}

Theorem \ref{thm:semigroup} associates to every exchangeable, cadlag
process $(X_t)_{t\geq0}$ on $[k]^{\zz{N}}$ a compatible random
semigroup $(Q_{s,t})_{0\leq s\leq t}$.  The random semigroup
$(Q_{s,t})_{0\leq s\leq t}$ is a two-parameter process on $\stochk$
that gives the transition probabilities of a time-inhomogeneous Markov
process on $[k]$.  The outcome of Theorem \ref{thm:semigroup} can be
viewed as a de Finetti-type characterization of exchangeable Markov
processes $(X_t)_{t\geq0}$ on $[k]^{\zz{N}}$ with cadlag sample paths.
In the special case when $(X_t)_{t\geq0}$ is a Feller process, the
cut-and-paste representation of $(X_t)_{t\geq0}$, as proven in Theorem
2.6 of \cite{Crane2012}, implies a description of the compatible
random semigroup as a {\em L\'evy matrix flow}.

\begin{defn}\label{def:flow}
A {\em L\'evy matrix flow} is a collection $(Q_{s,t})_{0\leq s\leq t}$
of stochastic matrices such that
\begin{itemize}
	\item[(i)] for every $s<t<u$, $Q_{s,u}=Q_{s,t}Q_{t,u}$
a.s.; \item[(ii)] for $s<t$, the law of $Q_{s,t}$ depends only on
$t-s$ 
\item [(iii)] for $s_1<s_2<\cdots<s_n$, the matrices
$Q_{s_1,s_2},Q_{s_2,s_3},\ldots,Q_{s_{n-1},s_n}$ are independent;
and, \item[(iv)] $Q_{0,0}=I_k$, the $k\times k$ identity matrix, and
$Q_{0,t}\stackrel{P}{\rightarrow} I_k$ as $t\downarrow0$.
\end{itemize}
\end{defn}
In Definition \ref{def:flow}(iv), convergence is with respect to the
total variation metric on $\stochk$: 
\[\xnorm{Q-Q'}_{TV}:=\sum_{i,j=1}^k|Q_{ij}-Q'_{ij}|,\quad Q,Q'\in\stochk.\]

Under the product-discrete topology on $[k]^{\zz{N}}$, the combination
of exchangeability and the Feller property is equivalent to the
combination of exchangeability and consistency under subsampling; see discussion
in Section \ref{section:introduction} and reference to \cite{Billingsley1995}.
Hence, under the additional assumption that $(X_t)_{t\geq0}$ is
Feller, each restriction $(X_t^{[n]})_{t\geq0}$ of $(X_t)_{t\geq0}$ to
$[k]^{[n]}$ is {\em unconditionally} a time-homogeneous Markov chain
on $[k]^{[n]}$, for every $n\in\mathbb{N}$.  When $(X_t)_{t\geq0}$ is
Feller, so is its projection $(Y_t)_{t\geq0}$ into $\simplexk$.  The
following observation follows directly from the discussion in
\cite{Crane2012}.

\begin{prop}\label{prop:flow}
Let $(X_t)_{t\geq0}$ be an exchangeable Feller process on
$[k]^{\zz{N}}$.  Then its compatible random semigroup
$(Q_{s,t})_{0\leq s\leq t}$ is a L\'evy matrix flow corresponding to a
unique Feller process on $\stochk$.  That is, there exists a unique
Feller process $(S_t)_{t\geq0}$ on $\stochk$ such that, for all $0\leq
s\leq t$,
\begin{itemize}
	\item $Q_{0,t}=S_t$,
	\item $Q_{s,t}\equalinlaw S_{t-s}$ and
	\item $Q_{s,t}$ is measurable with respect to $\sigma\langle S_{r}\rangle_{s\leq r\leq t}$.
\end{itemize}
\end{prop}


\begin{thebibliography}{1}

\bibitem{Berestycki2004}
J.~Berestycki.
\newblock Exchangeable fragmentation-coalescence processes and their
  equilibrium measures.
\newblock {\em Electron. J. Probab.}, 9:no. 25, 770--824 (electronic), 2004.

\bibitem{Bertoin2006}
J.~Bertoin.
\newblock {\em Random fragmentation and coagulation processes}, volume 102 of
  {\em Cambridge Studies in Advanced Mathematics}.
\newblock Cambridge University Press, Cambridge, 2006.

\bibitem{billingsley}
P.~Billingsley.
\newblock {\em Convergence of Probability Measures}.
\newblock John Wiley
  \& Sons Inc., New York,  1968.

\bibitem{Billingsley1995}
P.~Billingsley.
\newblock {\em Probability and measure}.
\newblock Wiley Series in Probability and Mathematical Statistics. John Wiley
  \& Sons Inc., New York, third edition, 1995.
\newblock A Wiley-Interscience Publication.

\bibitem{Crane2012}
H.~Crane.
\newblock Homogeneous cut-and-paste processes.
\newblock {\em Manuscript. Submitted.}, 2012.

\bibitem{daley-gani}
D. J. Daley and J. Gani.
\newblock \emph{Epidemic modelling: an introduction.}
\newblock Cambridge Studies in Mathematical Biology, 15. 
\newblock Cambridge University Press, Cambridge, 1999. 


\bibitem{dolgoarshinnykh-lalley} R.~Dolgoarshinnykh and S.~Lalley.
\newblock Critical scaling for the SIS stochastic epidemic.
\newblock \emph{J. Appl. Probab.}, 43, 892–898, 2006.

\bibitem{ethier-kurtz} S.~Ethier and T.~Kurtz.
\newblock \emph{Markov processes}.
\newblock Wiley Series in Probability and Mathematical Statistics:
              Probability and Mathematical Statistics. John Wiley
  \& Sons Inc., New York, 1986.

\bibitem{kurtz} T.~Kurtz.
\newblock Solutions of ordinary differential equations as limits of pure
              jump {M}arkov processes. 
\newblock \emph{J. Appl. Probability}, 7, 49--58, 1970.

\bibitem{hewitt-savage}
Hewitt and Savage.
\newblock Symmetric measures on cartesian products.
\newblock {\em Trans. Amer. Math. Soc.}, 80:470--501, 1955.

\bibitem{Kingman1982}
J.~F.~C. Kingman.
\newblock The coalescent.
\newblock {\em Stochastic Process. Appl.}, 13(3):235--248, 1982.

\bibitem{peres} D. Levin, M. Luczak, and Y. Peres.
\newblock Glauber dynamics for the mean-field Ising model: cut-off,
critical power law, and metastability.  
\newblock \emph{Probab. Theory Related Fields} 146,  223–-265, 2010.


\bibitem{Pitman2005}
J.~Pitman.
\newblock {\em Combinatorial stochastic processes}, volume 1875 of {\em Lecture
  Notes in Mathematics}.
\newblock Springer-Verlag, Berlin, 2006.
\newblock Lectures from the 32nd Summer School on Probability Theory held in
  Saint-Flour, July 7--24, 2002, With a foreword by Jean Picard.

\end{thebibliography}
\end{document}